\documentclass{amsart}
\usepackage{times,mathrsfs,color,comment}
\usepackage{amssymb,amsmath}

\numberwithin{equation}{section}

\newtheorem{thm}{Theorem}[section]

\newtheorem{lem}[thm]{Lemma}

\begin{document}

\title[Burkholder-Gundy-Davis Inequality in Variable Exponent Spaces]
{Burkholder-Gundy-Davis Inequality in Martingale \\Hardy Spaces with Variable Exponent}
\date{May 17, 2014}

\author[P. Liu]{Peide Liu}
\address{School of Mathematics and Statistics, Wuhan
University, Wuhan 430072, P.R.CHINA}
\email {pdliu@whu.edu.cn}

\author[M. Wang]{Maofa Wang}
\address{School of Mathematics and Statistics, Wuhan
University, Wuhan 430072, P.R.CHINA}
\email{mfwang.math@whu.edu.cn}

\thanks{This project was supported by the  NSFC (11471251, 11271293)}

\newcommand{\C}{\mathbb{C}}
\newcommand{\Cn}{\mathbb{C}^n}
\newcommand{\Cm}{\mathbb{C}^m}
\newcommand{\Sn}{\mathbb{S}_n}
\newcommand{\Sm}{\mathbb{S}_m}
\newcommand{\D}{\mathbb{D}}
\newcommand{\T}{\mathbb{T}}
\newcommand{\Bm}{\mathbb{B}_m}
\newcommand{\Bn}{\mathbb{B}_n}
\newcommand{\B}{\mathbb{B}_1}
\newcommand{\K}{\mathcal D_\beta^n}

\begin{abstract}
In this paper, the classical Dellacherie's theorem about stochastic
process is extended to variable exponent Lebesgue spaces. As its
applications, we obtain variable exponent analogues of several
famous inequalities in classical martingale theory, including
convexity lemma, Burkholder-Gundy-Davis' inequality and Chevalier's
inequality. Moreover, we investigate some other equivalent relations
between variable exponent martingale Hardy spaces.
\end{abstract}

\maketitle

\section{Introduction}

Due to their important role in elasticity, fluid dynamics, calculus
of variations, differential equations and so on, Musielak-Orlicz
spaces and their special case, variable exponent Lebesgue spaces
have been got more and more attention in modern analysis and
functional space theory. In particular, Musielak-Orlicz spaces were
studied by Orlicz and Musielak \cite{mu}. Hudzik \cite{h2} studied
some geometry properties of Musielak-Orlicz spaces. Kovacik and
Rakosnik \cite{k}, Fan and Zhao \cite{f} investigated various
properties of variable exponent Lebesgue spaces and Sobolev spaces.
Diening \cite{d3} and Cruz-Uribe et el \cite{c3,c4} proved the
boundedness of Hardy-Littlewood maximal operator on variable
exponent Lebesgue function spaces $L^{p(x)}(R^n)$ under the
conditions that the exponent $p(x)$ satisfies so called
log-H$\ddot{o}$lder continuity and decay restriction. Many other
authors studied its applications to harmonic analysis and some other
subjects.

As we well known, the situation of martingale spaces is different
from function spaces. For example, the log-H$\ddot{o}$lder
continuity of a measurable function on a probability space can not
be defined. Moreover, generally speaking, the "good $-\lambda$"
inequality method used in classical martingale theory can not be
used in variable exponent case. However, recently, variable exponent
martingale spaces have been paid more attention too. Among others,
Aoyama \cite{a} proved some inequalities under the condition that
the exponent $p$ is $\Sigma_0-$measurable. Nakai and Sadasue
\cite{n1} pointed out that the $\Sigma_0-$measurability is not
necessary for the boundedness of Doob's maximal operator, and proved
that the boundedness holds when every $\sigma-$algebra is generated
by countable atoms.

The aim of this paper is to establish some variable exponent
analogues of several famous inequalities in classical martingale
theory. By extending Dellacherie's theorem to variable exponent case
we obtain convexity Lemma and Burkholder-Gundy-Davis' inequality and
Chevalier's inequality for variable exponent martingale Hardy
spaces. Then we investigate some equivalent relations between
several variable exponent martingale Hardy spaces, specially, we
prove that the two predictable martingale spaces
$\mathcal{D}_{p(\cdot)}$ and $\mathcal{Q}_{p(\cdot)}$ are equivalent
and under regular condition the five martingale Hardy spaces
$H^*_{p(\cdot)},~ H^S_{p(\cdot)},~ H^s_{p(\cdot)}, ~
\mathcal{D}_{p(\cdot)}$ and $\mathcal{Q}_{p(\cdot)}$ with variable
exponent $1\leq p^-\leq p^+<\infty$ are equivalent (for their
definitions, see below).

\vskip .3cm

Let $(\Omega,\Sigma, \mu)$ be a non-atomic complete probability
space, $L^0(\Omega)$ the set of all measurable functions (i.e. r.v.)
on $\Omega$, and $E$ the expectation with respect to $\Sigma.$ We
say that $p\in \mathcal{P},$ if $p\in L^0(\Omega)$ with $1\leq
p(\omega)\leq \infty.$ For $p\in \mathcal{P},$ denote
$\Omega_\infty=\{\omega\in \Omega, p(\omega)=\infty\},$ and define
variable exponent Lebesgue space as follows:
$$L^{p(\cdot)}=\{u\in L^0(\Omega): \exists \gamma>0, ~\rho_{p(\cdot)}(\gamma u)
<\infty\},$$
where the modular
\begin{align}\label{1.1} \rho_{p(\cdot)}(u)=\int_{\Omega\setminus\Omega_\infty}|u(\omega)|^{p(\omega)}d\mu+ ess\sup_{\omega\in\Omega_\infty}
|u(\omega)|.\end{align}
For every $u\in L^{p(\cdot)},$
its Luxemburg norm is defined by
\begin{align}\label{1.2}
\|u\|_{p(\cdot)}=\inf \{\gamma>0: \rho_{p(\cdot)}(\frac{u}{\gamma})\leq1\}.
\end{align}
We denote by $p^+$ and $p^-$ the below index and upper index of $p,$
i.e.,
$$p^-= {ess\inf} _{\omega\in\Omega} \,p(\omega), \ \ \ \  p^+=
{ess\sup} _{\omega\in\Omega}  \, p(\omega), $$ and $p$'s conjugate
index is $ p'(\omega)$, i.e., $\frac{1}{p(\omega)}+
\frac{1}{p'(\omega)}=1. $

Here we mention some basic properties of $L^{p(\cdot)},$ the proofs
of which are standard and similar to classical function spaces, for
example, see \cite{f,k}.

\begin{lem}
\label{lem11} ~~   Let $p\in \mathcal{P}$ with $p^+<\infty,$
then\begin{itemize}

\item[(1)]~ $\rho_{p(\cdot)}(u)<1 ~(=1, >1) $ if and only if $
\|u\|_{p(\cdot)}<1 ~(=1, >1).$
\item[(2)]~ $\rho_{p(\cdot)}(\frac{u}{\|u\|_{p(\cdot)}})=1, \forall u\in
L^{p(\cdot)}$ with $0<\|u\|_{p(\cdot)}<\infty.$
\item[(3)]~ $\rho_{p(\cdot)}(u)\leq \|u\|_{p(\cdot)}$, if $
\|u\|_{p(\cdot)}\leq1.$
\item[(4)]~ $(L^{p(\cdot)},\|\cdot\|_{p(\cdot)})$ is a Banach space.
\item[(5)]~ If $u\in L^{p(\cdot)}, v\in L^{p'(\cdot)},$ then
\begin{align}\label{1.3}|Eu v |\leq C\|u\|_{p(\cdot)}\|v\|_{p'(\cdot)},\end{align}\end{itemize}
where $C$ is a positive constant depending only on $p.$
\begin{itemize}\item[(6)]~ If $u_n\in L^{p(\cdot)},$ then
$\|u_n-u\|_{p(\cdot)}\rightarrow0$ if and only if
$\rho_{p(\cdot)}(u_n-u)\rightarrow0.$
\end{itemize}
\end{lem}

\begin{lem}
\label{lem12} Let $p\in \mathcal{P}$ and $s>0$ such that $sp^-\geq
1,$ then
\begin{align}\label{1.4}  \||u|^s\|_{p(\cdot)}=
\|u\|^s_{sp(\cdot)}. \end{align}
\end{lem}~~ {\it
}

\begin{lem}
\label{lem13} ~~ {\it Let $p, q\in \mathcal{P},$ then
$L^{p(\cdot)}\subset L^{q(\cdot)}$ if and only if $p(\omega)\geq
q(\omega)~ a.e.$, and in this case the embedding is continuous with
\begin{align}\label{1.5} \|f\|_{q(\cdot)}\leq 2\|f\|_{p(\cdot)},~~ \forall f\in
L^{p(\cdot)}.\end{align} }\end{lem}

Let us fix some
notation in martingale theory.

Let $(\Sigma_n)_{n\geq 0}$ be a stochastic basis, i.e., a
nondecreasing sequence of sub-$\sigma$-algebras of $\Sigma$ with
$\Sigma=\bigvee\Sigma_n,$ $f=(f_n)_{n\geq0}$ a martingale adapted to
$(\Sigma_n)_{n\geq0}$ with its difference sequence $(df_n)_{n\geq0},
$ where $df_n=f_n-f_{n-1}$ (with convention $f_{-1}\equiv0$ and $
\Sigma_{-1}=\{\Omega,\emptyset\}).$ We denote by $E_n$ the
conditional expectation with respect to $\Sigma_n.$ For a martingale
$f=(f_n)_{n\geq0},$ we define its maximal function, square function
and conditional square function as usual:
$$f^*=\sup_{n\geq0}|f_n|,~~S(f)=(\sum_{n=0}^\infty|df_n|^2)^{\frac{1}{2}},~~s(f)=(\sum_{n=0}^\infty E_{n-1}|df_n|^2)^{\frac{1}{2}}.$$
For $p\in \mathcal{P},$ the variable exponent martingale Lebesgue
space $L^{p(\cdot)}$ and the martingale Hardy spaces
$H^*_{p(\cdot)}, H^S_{p(\cdot)}$ and $H^s_{p(\cdot)}$ are defined as
follows:
$$~L^{p(\cdot)}=\{f=(f_n): \|f\|_{p(\cdot)} = \sup\|f_n\|_{p(\cdot)}<\infty\}, $$
$$H^*_{p(\cdot)}=\{f=(f_n): \|f\|_{H^*_{p(\cdot)}}~ =
~\|f^*\|_{p(\cdot)}<\infty ~\}, ~~$$
$$H^S_{p(\cdot)}=\{f=(f_n): \|f\|_{H^S_{p(\cdot)}}=\|S(f)\|_{p(\cdot)}<\infty\}, $$
$$H^s_{p(\cdot)}=\{f=(f_n): \|f\|_{H^s_{p(\cdot)}}=\|s(f)\|_{p(\cdot)}<\infty\}. $$

The structure of this paper is as follows. After some preliminaries
about variable exponent Lebesgue spaces over a probability space, in
section 2 we mainly deal with the extension of Dellacherie's theorem
and the convexity lemma to variable exponent case. In section 3 we
establish the variable exponent analogues of Burkholder-Gundy-Davis'
inequality and Chevalier's inequality. In the last section we first
prove the equivalence between two martingale spaces with predictable
control, then investigate some equivalent relations between five
variable exponent martingale spaces under regular condition.

Through this paper, we always denote by $C$ some positive constant,
it may be different in each appearance, and denote by $C_{p(\cdot)}$
a constant depending only on $p.$ Moreover, we say that two norms on
$X$ are equivalent, if the identity is continuous in double
directions, i.e., there is a constant $C>0$ such that
$$C^{-1}\|u\|_1\leq \|u\|_2 \leq C\|u\|_1, ~~ \forall u\in X.$$

\vskip .3cm

\section{Some lemmas}\label{LFM}

In this section we prove some lemmas, which will be needed in the
sequel.

\begin{lem}
\label{21}~~   Let $p\in \mathcal{P}$ with $1\leq
p(\omega)\leq\infty, $ then every martingale or nonnegative
submartingale $f=(f_n)$ satisfying $\sup\|f_n\|_{p(\cdot)}<\infty$
converges a.e. to a measurable function $f_\infty\in L^{p(\cdot)}.$
\end{lem}\vskip .1cm

 \begin{proof}~~ Since $p(\omega)\geq1,$ from Lemma 1.3 we have
$$\sup_{n\geq0} \|f_n\|_1 \leq 2 \sup_{n\geq0}
\|f_n\|_{p(\cdot)}<\infty. $$ By Doob's martingale convergence
theorem, $f_n\rightarrow f_\infty~ a.e.,$ by Fatou lemma,
$f_\infty\in L^{p(\cdot)}.$ \end{proof}

In classical theory, Dellacherie exploited a special approach to
prove convex $\Phi-$function inequalities for martingales. It was
first formulated in \cite{d2}, also see \cite{l3}. The following
lemma generalizes Dellacherie theorem to variable exponent case.

\begin{lem}
\label{23}~~  Let $ p\in \mathcal{P}$ with $1\leq p^-\leq
p^+<\infty,~v$ be a non-negative r.v., $(u_n)_{n\geq0}$ a
nonnegative, nondecreasing adapted sequence satisfying
\begin{align}\label{2.1}E(u_\infty-u_{n-1}|\Sigma_n)\leq E(v|\Sigma_n),~~\forall n\geq0,\end{align}
or a nonnegative, nondecreasing predictable sequence with $u_0=0$
and
\begin{align}\label{2.2}E(u_\infty-u_n|\Sigma_n)\leq E(v|\Sigma_n), ~~\forall n\geq0.\end{align}
 then\vskip .5cm

\begin{enumerate}
\item[(1)]~ $E u_\infty^p \leq  p^+ E v u_\infty^{p-1},$

\item[(2)]~ $\rho_{p(\cdot)}(u_\infty)\leq C\rho_{p(\cdot)}(v), $

\item[(3)]~ $\|u_\infty\|_{p(\cdot)}\leq C \|v\|_{p(\cdot)},$ if $v\in
L^{p(\cdot)},$\end{enumerate}
where $C$ is a positive constant depending only on $p.$
\end{lem}\vskip.1cm

\begin{proof}~~ $1^\circ$~ First we assume that $p$ is
$\Sigma_n-$measurable for some $n,$ then there is a simple function
sequence $\{s_i\}$  such that all $s_i$ are $\Sigma_n-$measurable,
$s_i\geq1$ and $s_i\uparrow p.$

Let $s_i=\sum_{j=1}^{J_i} a_{i,j}\chi_{A_{i,j}},$ where $A_{i,j}\in
\Sigma_n$ with $A_{i,j}\cap A_{i,j'}=\emptyset, j\neq j',
\bigcup_{j=1}^{J_i}A_{ij}=\Omega.$ For each $A_{i,j},$ take
$\Phi(t)=t^{a_{i,j}}$, replace $(u_n)_{n\geq1}$ and $v$ by
$(u'_m)_{m\geq 0}$ and $v',$ respectively. Here,
$u'_m=u_{m+n}\chi_{A_{i,j}}, \Sigma'_m=\Sigma_{m+n}$, $v'= v
\chi_{A_{i,j}}.$  In view of the $\Sigma_n$-measurability of
$A_{i,j},$ inequality (2.1) becomes
\begin{align}\label{2.3}
\begin{aligned}
E(u'_\infty-u'_{m-1}|\Sigma'_m)&=E(u_\infty\chi_{A_{i,j}}-u_{m+n-1}\chi_{A_{i,j}}|\Sigma_{m+n})\\
&=E(u_\infty-u_{m+n-1}|\Sigma_{m+n})\chi_{A_{i,j}}\\
&\leq E(v|\Sigma_{m+n})\chi_{A_{i,j}}\\
&= E(v\chi_{A_{i,j}}|\Sigma_{m+n})=E(v'|\Sigma'_m),~~\forall~m\geq0.
\end{aligned}\end{align}
Similarly, (2.2) becomes
\begin{align}\label{2.4}E(u'_\infty-u'_m|\Sigma'_m)\leq E(v'|\Sigma'_m),~~\forall~m\geq0.\end{align}
By classical Dellacherie theorem, we get
$$Eu_\infty^{a_{i,j}}\chi_{A_{i,j}}\leq
E a_{i,j} v u_\infty^{a_{i,j}-1}\chi_{A_{i,j}},$$ and
\begin{align}
\label{2.5} Eu_\infty^{s_i}=\sum_{j=1}^{J_i}E
u_\infty^{a_{i,j}}\chi_{A_{i,j}} \leq \sum_{j=1}^{J_i} a_{i,j}Ev
u_\infty^{a_{i,j}-1}\chi_{A_{i,j}}\leq p^+ E vu_\infty^{s_i-1}.
\end{align}
 If $Ev u_\infty^{p-1}<\infty,$ it is
clear that $$ v u_\infty^{s_i-1}=
v|u_\infty\chi_{\{u_\infty<1\}}|^{s_i-1}+
v|u_\infty\chi_{\{u_\infty\geq 1\}}|^{s_i-1},$$ Lebesgue dominated
convergence theorem and Levi monotonic convergence theorem give
$Evu_\infty^{s_i-1}\rightarrow Ev u_\infty^{p-1},$ as $
i\rightarrow\infty.$ Similarly, from
$$|u_\infty|^{s_i}=|u_\infty\chi_{\{u_\infty<1\}}|^{s_i}
+|u_\infty\chi_{\{u_\infty\geq1\}}|^{s_i}, $$ we get
$Eu_\infty^{s_i}\rightarrow Eu_\infty^{p},$ as $i\rightarrow\infty.$
By taking limit on both sides of (2.5), we obtain $Eu_\infty^p \leq
p^+Ev u_\infty^{p-1},$ this is (1).

Now suppose that $p\in\mathcal{P}$ only,  we claim that there is a
sequence $\{p_k\}$ of simple functions such that $p_k$ is
$\Sigma_{n_k}$-measurable for some $n_k$ and $p_k\uparrow p,
n_k\uparrow\infty$ as $k\uparrow\infty.$ Indeed, we first take a
simple function sequence $\{g_k\},$ which is $\Sigma$-measurable
such that $g_k\uparrow p.$ Due to
$\Sigma=\sigma(\cup_{n=1}^\infty\Sigma_n),$ for every $A\in\Sigma,$
there is a sequence $A_k\in \cup_{n=1}^\infty\Sigma_n$ such that
$\mu(A\bigtriangleup A_k)\rightarrow0.$ Since $(\Sigma_n)$ is
increasing, for every $g_k,$ there is a simple function $g'_k$ such
that $g'_k$ is $\Sigma_{n_k}$-measurable, $g'_k\leq p$ and
$\mu\{g_k\neq g'_k\}<2^{-k}, $ so we have $g'_k\rightarrow p,$ a.e.
Let $p_k=g'_1\vee \cdots \vee g'_k,$ then $p_k\uparrow p,$ a.e. From
previous proof, (1) holds for every $p_k.$   We then obtain  (1) for
the general case by taking limit.

$2^\circ$~ Notice that for $p, q\in \mathcal{P}$ with
$\frac{1}{p(\omega)}+\frac{1}{q(\omega)}=1,$ by Young inequality we
have
$$p^+ba^{p-1}\leq \frac{(p^+b)^p}{p}+\frac{a^p}{q}, ~~ a^p=\frac{a^p}{p}+\frac{a^p}{q}, ~~ \forall a, b >0.$$
Letting $a=u_\infty(\omega), b=v(\omega)$ in these inequalities and
taking integrals on both sides, then (2) immediately follows form
(1).

$3^\circ$~ To prove (3), we may assume $\|u_\infty\|_{p(\cdot)}=1.$
Since
$\rho_{p'(\cdot)}(u_\infty^{p-1})=\rho_{p(\cdot)}(u_\infty)=1,$
Lemma 1.1(5) and the proof above of (1)  show that
$$1= \rho_{p(\cdot)}(u_\infty)=E u_\infty^p
\leq p^+Evu_\infty^{p-1}\leq C
\|v\|_{p(\cdot)}\|u_\infty^{p-1}\|_{p'(\cdot)}= C
\|v\|_{p(\cdot)},$$ so $\|u_\infty\|_{p(\cdot)}\leq C
\|v\|_{p(\cdot)}.$ The proof is complete. \end{proof}

The following lemma is so called convexity lemma whose classical
version belongs to Burkholder, Davis and Gundy, see \cite{b2}.

\begin{lem}
\label{24}~~   Let $p\in \mathcal{P}$ with $p^+<\infty,$
$(\xi_n)_{n\geq0}$ be a non-negative r.v. sequence, then there is a
constant $C=C_{p(\cdot)}>0$ such that
\begin{align}\label{2.6}\|\sum_{n=1}^\infty E_n\xi_n\|_{p(\cdot)}\leq C \|\sum_{n=1}^\infty
\xi_n\|_{p(\cdot)}.\end{align}
\end{lem}

\begin{proof} ~~
Let $v_n=\sum_{k=1}^n \xi_k,~u_n=\sum_{k=1}^n E_k\xi_k,$ for any $n$
we have
$$E_n(u_\infty-u_{n-1})=E_n(v_\infty-v_{n-1})\leq E_nv_\infty,$$
so (2.6) follows from Lemma 2.2(3). \end{proof}

\begin{lem}
\label{25}~~   Let $p\in \mathcal{P}$ with $2\leq p\leq
p^+<\infty,$ then there is a $C=C_{p(\cdot)}$ such that for every
martingale $f=(f_n),$
\begin{align}\label{2.7}\|s(f)
\|_{p(\cdot)}\leq
C\|S(f)\|_{p(\cdot)}.\end{align}
\end{lem}\vskip .1cm

\begin{proof} ~~ Let $v_n=S_n(f)^2=\sum_{k=1}^n
|df_k|^2,~u_n=s_n(f)^2=\sum_{k=1}^n E_{k-1}|df_k|^2,$ then by Lemmas
1.2 and 2.3 we obtain
$$\|s(f)\|_{p(\cdot)}=\|s(f)^2\|^{\frac{1}{2}}_{\frac{p(\cdot)}{2}}\leq C\|S(f)^2\|^{\frac{1}{2}}_{\frac{p(\cdot)}{2}}=C\|S(f)\|_{p(\cdot)}.$$
This is desired. \end{proof}


\section{Burkholder-Gundy-Davis inequality}

Let us first extend classical Burkholder-Gundy-Davis inequality to
variable exponent case, which is one of the most fundamental
theorems in martingale theory, see \cite{b2}.

\begin{thm}\label{31}~~   Let $p\in \mathcal{P}$ with $1\leq p^-\leq
p^+<\infty,$ then there is a $C=C_{p(\cdot)}$ such that for every
martingale $f=(f_n),$
\begin{align}\label{3.1} C^{-1}\|f^*\|_{p(\cdot)} \leq \| S(f) \|_{p(\cdot)}\leq
C\|f^*\|_{p(\cdot)}.\end{align}
\end{thm}~~

\begin{proof} ~~ Here we
 use Davis' method. For a martingale $f=(f_n),$ we define $g=(g_m),
g_m=f_{m+n}-f_{n-1}, ~ \Sigma'_m=\Sigma_{m+n}, m\geq 0$. Due to the
classical Burkholder-Gundy-Davis inequality (in conditioned
version), we have
\begin{align}\label{3.2}
\begin{aligned}E(f^*-f^*_{n-1}|\Sigma_n)&\leq E(\sup_{m\geq0}|f_{n+m}-f_{n-1}||\Sigma_n)\\
&=E(g^*|\Sigma'_0)\leq CE(S(g)|\Sigma'_0)\\
&=C E((S(f)^2-S_{n-1}(f)^2)^{\frac{1}{2}}|\Sigma_n)\\
&\leq C E(S(f)|\Sigma_n),
\end{aligned}
\end{align}
and
\begin{align}\label{3.3}
\begin{aligned}E(S(f)-S_{n-1}(f)|\Sigma_n)&\leq E((S(f)^2-S_{n-1}(f)^2)^{\frac{1}{2}}|\Sigma_n)\\
&\leq E(S(g)|\Sigma'_0) \leq CE(g^*|\Sigma'_0) \leq
CE(f^*|\Sigma_n).
\end{aligned}
\end{align}
Using Lemma 2.2, the inequality (3.1) follows from (3.2) and (3.3).
\end{proof}

Now we prove a more shaper inequality: Chevalier's inequality. For a
martingale $f=(f_n),$ consider the functions $M(f)$ and $m(f):$
$$M(f)=\sup_{n\geq0} M_n(f),~~ m(f)=\sup_{n\geq0} m_n(f),$$
where $$M_n(f)=f^*_n\vee S_n(f), ~~ m_n(f)=f^*_n\wedge S_n(f).$$

\begin{thm}\label{32}~~   Let $p\in \mathcal{P}$ with $1\leq p^-\leq
p^+<\infty,$ then there is a $C=C_{p(\cdot)}$ such that for every
martingale $f=(f_n),$
\begin{align}\label{3.4}\|M(f)\|_{p(\cdot)}\leq
C\|m(f)\|_{p(\cdot)}.\end{align}
\end{thm}\vskip .1cm

\begin{proof} ~~
We begin with a well known result (see Long \cite{l3}, Theorem
3.5.5.): let $g=(g_m)$ be as in the proof of Theorem 3.1 and $(D_n)$
a predictable control of difference sequence $(df_n),$ i.e. $(D_n)$
is an increasing adapted r.v. sequence with $|df_n|\leq
D_{n-1},n\geq0, $ then there is a $C>0$ such that
\begin{align}\label{3.5}E_0(M(f))\leq CE_0(m(f)+D_\infty).\end{align}

 Now for any fixed $n$, we have
\begin{align*}
M(f)-M_{n-1}(f)&\leq (f^*-f^*_{n-1})\vee (S(f)-S_{n-1}(f))\\
&\leq g^*\vee S(g)=M(g),
\end{align*}
and
$$m(g)=g^*\wedge S(g) \leq 2f^*\wedge S(f) \leq 2m(f).$$
Using (3.5) we get
\begin{align*}
E(M(f)-M_{n-1}(f)|\Sigma_n)&\leq E(M(g)|\Sigma'_0)  \\
 &\leq CE(m(g)+ D'_\infty|\Sigma'_0)\leq
 CE(m(f)+D_\infty|\Sigma_n).
\end{align*}
Where $D'$ is a predictable control of $g$ with  $D'=(D_m'),
D'_m=D_{m+n}, D'_\infty=D_\infty.$ Lemma 2.2 guarantees that the
following inequality holds:
\begin{align} \label{3.6} \|M(f)\|_{p(\cdot)}\leq
C\|m(f)+D_\infty\|_{p(\cdot)}.
\end{align}

Making $f$'s Davis decomposition $f=g+h,$ with $|dg_n|\leq
4d^*_{n-1} $ and $$\sum|dh_n|\leq 2d^*+2\sum
E_{n-1}(d^*_n-d^*_{n-1}),~~~~~~~~~~~~~~$$ where $d^*_n=\sup_{0\leq
k\leq n}|df_k|,$ we have the following estimate:
$$d^*\leq 2f^*\wedge S(f)\leq 2m(f), ~~ \ ~~ h^*\vee S(h)\leq \sum|dh_n|$$
and
\begin{align*}
 m(g) &\leq (f^*+h^*)\wedge(S(f)+S(h))\\
&\leq f^*\wedge S(f)+\sum|dh_n| =m(f)+\sum|dh_n|.
\end{align*}
From Lemma 2.3, we then have
$$\|\sum|dh_n|\|_{p(\cdot)}\leq C\|d^*\|_{p(\cdot)}+C\|\sum E_{n-1}(d^*_n-d^*_{n-1})\|_{p(\cdot)}\leq C\|d^*\|_{p(\cdot)}.$$
By using (3.6), we obtain
\begin{align*}\|M(f)\|_{p(\cdot)}&\leq \|M(g)\|_{p(\cdot)}+\|M(h)\|_{p(\cdot)}\\
&\leq C\|m(g)\|_{p(\cdot)}+C \|d^*\|_{p(\cdot)}+C\|m(f)\|_{p(\cdot)}\\
&\leq C\|m(f)\|_{p(\cdot)}.\end{align*}
This completes the proof.
\end{proof}

\vskip .3cm

 \section{Some equivalent relations
between martingale spaces}

 Let $p\in \mathcal{P}$ with $1\leq p\leq p^+<\infty,$
$\lambda=(\lambda_n)$ be a nonnegative and increasing adapted
sequence
 with $\lambda_\infty=\lim_{n\rightarrow\infty}\lambda_n\in
L^{p(\cdot)}.$ We denote by $\Lambda$ the set of all such sequences
and define two martingale spaces as follows:
$$~~~\mathcal{Q}_{p(\cdot)}=\{f=(f_n): \exists \lambda\in\Lambda, S_n(f)\leq \lambda_{n-1},
 \|f\|_{\mathcal{Q}_{p(\cdot)}}=\inf_{\lambda\in\Lambda}\|\lambda_\infty\|_{p(\cdot)}<\infty\},$$
$$~~~ \mathcal{D}_{p(\cdot)}=\{f=(f_n): ~ \exists \lambda\in\Lambda, ~|f_n|\leq
\lambda_{n-1}, ~
 \|f\|_{\mathcal{D}_{p(\cdot)}} = \inf_{\lambda\in\Lambda}\|\lambda_\infty\|_{p(\cdot)}<\infty\}.$$
It is easy to check that both two martingale spaces above are Banach
spaces, and as in the classical case the norms of
$\mathcal{Q}_{p(\cdot)}, \mathcal{D}_{p(\cdot)}$ can be reached, we
call such $\lambda$ an optimal predictable control of $f.$ We also
introduce the following martingale space $\mathcal{A}_{p(\cdot)}:$
$$\mathcal{A}_{p(\cdot)}=\{f=(f_n): \|f\|_{\mathcal{A}_{p(\cdot)}}=\|\sum_{n=0}^\infty |df_n|\|_{p(\cdot)}<\infty\}.$$

To prove the equivalence between $\mathcal{Q}_{p(\cdot)}$ and
$\mathcal{D}_{p(\cdot)}$, we first need the following theorem.

\begin{thm}\label{41}~~   Let $p\in \mathcal{P}$ with $1\leq p^-\leq
p^+<\infty,$ then there are $C=C_{p(\cdot)}$ such that the following
inequalities hold for every martingale $f=(f_n),$
\begin{align}\label{4.0} \| f
\|_{H^*_{p(\cdot)}} \leq \| f \|_{\mathcal{D}_{p(\cdot)}}, ~~~~ \ \
\ \ \ ~~~~ \| f \|_{H^S_{p(\cdot)}} \leq \| f
\|_{\mathcal{Q}_{p(\cdot)}},
\end{align}
 \begin{align}\label{4.1}~~~~~~\|f\|_{H^*_{p(\cdot)}} \leq C\|f\|_{\mathcal{Q}_{p(\cdot)}},~~~ \ \ ~~\|f\|_{H^S_{p(\cdot)}} \leq C\|f\|_{\mathcal{D}_{p(\cdot)}}.\end{align}
\end{thm} \vskip .1cm

 \begin{proof}~~ The two inequalities in (4.1) are obvious from their definitions. And the two inequalities in (4.2) come from (3.1) and (4.1).
\end{proof}

We now prove a Davis' decomposition theorem for martingales in
$H^S_{p(\cdot)}$ and $H^*_{p(\cdot)}.$

\begin{thm}\label{42}~~   Let $p\in \mathcal{P}$ with $1\leq p^-\leq
p^+<\infty,$ then
\begin{enumerate}
\item[(1)]~ Every $f=(f_n)\in H^S_{p(\cdot)}$ has a decomposition $f=g+h$
with $g\in \mathcal{Q}_{p(\cdot)}, h\in \mathcal{A}_{p(\cdot)}$ such
that
\begin{align}\label{4.2} \|g\|_{\mathcal{Q}_{p(\cdot)}} \leq
C\|f\|_{H^S_{p(\cdot)}},~~~\|h\|_{\mathcal{A}_{p(\cdot)}} \leq C \|
f \|_{H^S_{p(\cdot)}}.\end{align}
\item[(2)]~ Every $f=(f_n)\in H^*_{p(\cdot)}$ has a decomposition $f=g+h$
with $g\in \mathcal{D}_{p(\cdot)}, h\in \mathcal{A}_{p(\cdot)}$ such
that
\begin{align}\label{4.3} \|g\|_{\mathcal{D}_{p(\cdot)}} \leq C\|f\|_{H^*_{p(\cdot)}}, ~~~ \|h\|_{\mathcal{A}_{p(\cdot)}} \leq C\| f
\|_{H^*_{p(\cdot)}}. \end{align}
\end{enumerate}\end{thm}\vskip .1cm

\begin{proof} ~~ Here we
only prove (4.3), the proof of (4.4) is similar.

Let $\lambda=(\lambda_n)$ be an increasing control of
$(S_n(f))_{n\geq0}: |S_n(f)|\leq \lambda_n,  \lambda_\infty \in
L^{p(\cdot)}.$ Define
$$dh_n=df_n\chi_{\{\lambda_n>2\lambda_{n-1}\}}-E_{n-1}(df_n\chi_{\{\lambda_n>2\lambda_{n-1}\}}),$$
$$dg_n=df_n\chi_{\{\lambda_n\leq2\lambda_{n-1}\}}-E_{n-1}(df_n\chi_{\{\lambda_n\leq2\lambda_{n-1}\}})$$
and $h_n=\sum_{k=0}^ndh_k, g_n=\sum_{k=0}^ndg_k.$ It is clear that
bosh $(h_n)_{n\geq0}$ and $(g_n)_{n\geq0}$ are martingales and
$f_n=g_n+h_n, \forall n\geq0.$ As usual, we have
$$|df_k|\chi_{\{\lambda_k>2\lambda_{k-1}\}}\leq \lambda_k\chi_{\{\lambda_k>2\lambda_{k-1}\}}\leq 2\lambda_k-2\lambda_{k-1},$$
thus
$$\sum_{k=0}^\infty |dh_k|\leq 2\lambda_\infty+2\sum_{k=0}^\infty E_{k-1}(\lambda_k-\lambda_{k-1}).$$
Lemma 2.3 guarantees
\begin{align}\label{4.4}\|h\|_{\mathcal{A}_{p(\cdot)}} = \|\sum_{k=0}^\infty |dh_k|\|_{p(\cdot)}\leq C\|\lambda_\infty\|_{p(\cdot)}.\end{align}

On the other hand, since
$|df_k|\chi_{\{\lambda_k\leq2\lambda_{k-1}\}}\leq 2\lambda_{k-1}$
and $|dg_k|\leq 4\lambda_{k-1},$ thus
\begin{align*}
S_n(g)&\leq S_{n-1}(g)+|dg_n|\leq S_{n-1}(f)+S_{n-1}(h)+4\lambda_{n-1}\\
&\leq \lambda_{n-1}+2\lambda_{n-1}+2\sum_{k=0}^{n-1} E_{k-1}(\lambda_k-\lambda_{k-1})+4\lambda_{n-1},
\end{align*}
so $g \in \mathcal{Q}_{p(\cdot)},$ also due to Lemma 2.3,
\begin{align}\label{4.5}\|g\|_{\mathcal{Q}_{p(\cdot)}}
\leq \|7\lambda_\infty+2\sum_{k=0}^\infty
E_{k-1}(\lambda_k-\lambda_{k-1})\|_{p(\cdot)} \leq
C\|\lambda_\infty\|_{p(\cdot)}.\end{align}
 Taking
$\lambda_n=S_n(f),$ then (4.3) follows from (4.5) and (4.6).
\end{proof}

The following statement is about the equivalence between
$\mathcal{Q}_{p(\cdot)}$ and $\mathcal{D}_{p(\cdot)}$. Refer to Chao
and Long \cite{c1}, also see \cite{w} for the classical version.

\begin{thm}\label{43}~~   Let $p\in \mathcal{P}$ with $1\leq p^-\leq
p^+<\infty,$ then there is a $C=C_{p(\cdot)}>0$ such that for every
martingale $f=(f_n),$
\begin{align}\label{4.6} C^{-1}\|f\|_{\mathcal{D}_{p(\cdot)}}
\leq \|f\|_{\mathcal{Q}_{p(\cdot)}}\leq
C\|f\|_{\mathcal{D}_{p(\cdot)}}.\end{align}
\end{thm} \

\begin{proof}~~ Let $f=(f_n)\in \mathcal{D}_{p(\cdot)}$ and
$\lambda=(\lambda_n)$ be its optimal predictable control: $|f_n|\leq
\lambda_{n-1},
\|f\|_{\mathcal{D}_{p(\cdot)}}=\|\lambda_\infty\|_{p(\cdot)}$. Since
$$S_n(f)\leq S_{n-1}(f)+|df_n|\leq S_{n-1}(f)+2\lambda_{n-1},$$
namely, $(S_{n-1}(f)+2\lambda_{n-1})_{n\geq0}$ is a predictable
control of $(S_n(f))_{n\geq0},$ then $f\in \mathcal{Q}_{p(\cdot)}$
and it follows from (4.2) that
$$\|f\|_{\mathcal{Q}_{p(\cdot)}}\leq \|S(f)+2\lambda_\infty\|_{p(\cdot)}
\leq \|f\|_{H^S_{p(\cdot)}}+2\|\lambda_\infty\|_{p(\cdot)}\leq C\|f\|_{\mathcal{D}_{p(\cdot)}}.$$

Conversely, if $f=(f_n)\in \mathcal{Q}_{p(\cdot)}$ and
$\lambda=(\lambda_n)$ is its predictable control. Since $$|f_n|\leq
|f_{n-1}|+|df_n|\leq f^*_{n-1}+2\lambda_{n-1},$$ then $f\in
\mathcal{D}_{p(\cdot)}.$ Using (4.2) again we obtain
$$\|f\|_{\mathcal{D}_{p(\cdot)}}\leq \|f^*+2\lambda_\infty\|_{p(\cdot)}
\leq \|f\|_{H^*_{p(\cdot)}}+2\|\lambda_\infty\|_{p(\cdot)}\leq
C\|f\|_{\mathcal{Q}_{p(\cdot)}}.$$ The proof is complete.
\end{proof}

\begin{thm}\label{44}~~   If $p\in \mathcal{P}$ with $1\leq p^-\leq
p^+<\infty,$ then there is a $C=C_{p(\cdot)}$ such that for every
martingale $f=(f_n)$ with $f_0=0,$
\begin{align}\label{4.7}
\|f\|_{H^s_{p(\cdot)}}\leq C\|f\|_{\mathcal{D}_{p(\cdot)}},~~\|f\|_{H^S_{p(\cdot)}}
\leq C\|f\|_{\mathcal{D}_{p(\cdot)}}.\end{align}
\end{thm}\vskip .1cm

\begin{proof}~~ Here we use Garsia's idea which was used to prove theorem
4.1.2 in \cite{g}. Let $f\in \mathcal{D}_{p(\cdot)}$ and
$\lambda=(\lambda_n)$ be its optimal predictable control:
$\lambda_n$ is positive and increasing, $|f_n|\leq \lambda_{n-1}$
and $ \|\lambda_\infty\|_{p(\cdot)}=\|f\|_{\mathcal{D}_{p(\cdot)}}.$
Define
$$g_n= \sum_{i=1}^n\frac{df_i}{\sqrt{\lambda_{i-1}}},~~ \forall n\geq1. $$
A simple computation shows that
$$g_n= \sum_{i=1}^n\frac{f_i-f_{i-1}}{\sqrt{\lambda_{i-1}}}
=\frac{f_n}{\sqrt{\lambda_{n-1}}}+\sum_{i=1}^{n-1}\frac{f_i}{\sqrt{\lambda_{i-1}\lambda_i}}(\sqrt{\lambda_i}-\sqrt{\lambda_{i-1}}),$$
and
\begin{align}\label{4.8}|g_n|\leq 2\sqrt{\lambda_{n-1}},~~ g^*\leq
2\sqrt{\lambda_\infty}, ~~ Eg^{*2}\leq 4 E\lambda_\infty<\infty.
\end{align}
So $g=(g_n)$ is an $L^2$-bounded martingale, it
converges to
$g_\infty=\sum_{i=1}^\infty\frac{df_i}{\sqrt{\lambda_{i-1}}}$ a.e.
and in $L^2.$ Notice that
$$f_n= \sum_{i=1}^n \sqrt{\lambda_{i-1}} dg_i ,~~ \forall n\geq1, $$
then
\begin{align}\label{4.9}s^2_n(f)=\sum_{i=1}^n\lambda_{i-1}E_{i-1}|dg_i|^2\leq \lambda_{n-1} s^2_n(g).\end{align}
and
\begin{align}\label{4.10}S^2_n(f)\leq \lambda_{n-1} S^2_n(g).\end{align}
By H\"{o}lder inequality we have
\begin{align}\label{4.11}Es(f)^p\leq E\lambda_\infty^{\frac{p}{2}}s(g)^p
\leq
(E\lambda_\infty^p)^{\frac{1}{2}}(Es(g)^{2p})^{\frac{1}{2}}.\end{align}
Using Lemmas 2.2(2), 2.4 and inequalities (3.3), (4.9) we obtain
\begin{align*}Es(f)^p &\leq C(E\lambda_\infty^p)^{\frac{1}{2}}(ES(g)^{2p})^{\frac{1}{2}}\\
&\leq C(E\lambda_\infty^p)^{\frac{1}{2}}(Eg^{*2p})^{\frac{1}{2}}\\
&\leq C(E\lambda_\infty^p)^{\frac{1}{2}}(E\lambda_\infty^p)^{\frac{1}{2}}=C (E\lambda_\infty^p),\end{align*}
this implies the first inequality of (4.8).

A similar argument gives the second inequality of (4.8).
\end{proof}

Now we consider some equivalent relations between five martingale
spaces under regular condition. In \cite{w}, Weisz called a
martingale $f=(f_n)$ is previsible, if there is a real number $R>0$
such that
\begin{align}\label{4.12}|df_n|\leq RE_{n-1}|df_n|, ~~ \forall n\geq 0,\end{align}  and proved that
if it holds for all martingale with the same constant $R$, then the
stochastic basis $(\Sigma_n)$ is regular (refer to
Garsia \cite{g} for its definition). He also proved if $(\Sigma_n)$ is regular, then all
the spaces $H^*_p,~ H^S_p,~ H^s_p, ~ \mathcal{D}_p$ and
$\mathcal{Q}_p $ are equivalent for $0<p<\infty.$ The following
theorem is the  variable exponent analogues.

\begin{thm}\label{45}~~   If $p\in \mathcal{P}$ with $1\leq p^-\leq
p^+<\infty$ and the stochastic basis $(\Sigma_n)$ is regular, then
the martingale spaces $H^*_{p(\cdot)},~ H^S_{p(\cdot)},~
H^s_{p(\cdot)}, ~ \mathcal{D}_{p(\cdot)}$ and
$\mathcal{Q}_{p(\cdot)}$ are equivalent.
\end{thm}\vskip .1cm

\begin{proof}~~ Under the regular condition,  we notice  that
$$S_n(f)\leq S_{n-1}(f)+|df_n|\leq S_{n-1}(f)+ RE_{n-1}|df_n| \leq S_{n-1}(f)+ RE_{n-1}S_n(f),$$
that is, $(S_{n-1}(f)+ RE_{n-1}S_n(f))_{ n\geq0}$ is a predictable
control of $(S_n(f))_{n\geq0}.$  Since $$E_{n-1}S_n(f)\leq
S_{n-1}(f)+E_{n-1}(S_n(f)-S_{n-1}(f)),$$  Lemma 2.3 gives
$$\|f\|_{\mathcal{Q}_{p(\cdot)}}\leq 2\|S(f)\|_{p(\cdot)}
+ R\|\sum_{n=0}^\infty E_{n-1}(S_n(f)-S_{n-1}(f))\|_{p(\cdot)}\leq C
\|S(f)\|_{p(\cdot)},$$ where $C$ is some constant depending only on
$p$ and $R$. Then it follows from Theorems 3.1 and 4.3 that
\begin{align}\label{4.13}\|f\|_{H^S_{p(\cdot)}}\leq C\|f\|_{H^*_{p(\cdot)}}
\leq C\|f\|_{\mathcal{D}_{p(\cdot)}}\leq
C\|f\|_{\mathcal{Q}_{p(\cdot)}}\leq C\|f\|_{H^S_{p(\cdot)}}.\end{align}

  It remains to
prove
\begin{align}\label{4.14}C^{-1} \|f\|_{H^s_{p(\cdot)}} \leq
 \|f\|_{H^S_{p(\cdot)}}\leq C\|f\|_{H^s_{p(\cdot)}}.\end{align}
The first inequality comes from Theorem 4.4 and (4.13). Due to the
regularity, we have $S_n(f)\leq R s_n(f)$ and $S(f)\leq R s(f),$ so
the second inequality follows directly. The proof is complete.
\end{proof}

\medskip




\begin{thebibliography}{99}


\bibitem{a} H.Aoyama, {\it Lebesgue spaces with variable exponent on a
probability space, } Hiroshima Math. J. 39(2009), 207-216.

\bibitem{b1} D.L.Burkholder, {\it Distribution function inequalities for martingales,}
Annals of Prob. 1(1973), 19-42.

\bibitem{b2} D.L.Burkholder, B.J.Davis, R.F.Gundy, {\it Integral inequalities for convex functions of
operators on martingales, } Proc. Sixth Berkeley Symp. Math. Stat.
and Prob. 1972, 223-240.

\bibitem{c1} J.A.Chao, R.L.Long, {\it Martingale transforms with unbounded multipliers,}
Proc. Amer. Math. Soc. 114(1992), 831-838.

\bibitem{c2} L.Chevalier, {\it Un nouveau type d'inegalites pour les martingales discretes,}
Z. Wahrs. Verw. Gebiete 49(1979), 249-255.

\bibitem{c3} D. Cruz-Uribe, A.Fiorenza, J.M.Martell, C.
Perez, {\it The boundedness of classical operators on variable $L^p$
spaces,} Ann. Acad. Sci. Fenn. Math. 31(2006), 239-264.

\bibitem{c4} D. Cruz-Uribe, A.Fiorenza, C.J. Neugebauer, {\it The
maximal function on variable $L^{p(\cdot)}$ spaces,} Ann. Acad. Sci.
Fenn. Math. 28(2003), 223-238;  29(2004), 247-249.

\bibitem{d1} B.J.Davis, {\it On the integrability of the martingale square function, }
Israel J. Math. 8(1970), 187-190.

\bibitem{d2} C. Dellacherie, {\it Inegalites de convexite pour les processus croissants
et les sousmartingales,}  Sem. Prob. 13, LNM 721, 1979, 371-377.

\bibitem{d3} L. Diening, {\it Maximal function on generalized Lebesgue
spaces $L^{p(\cdot)}$}, Math. Inequal. Appl. 7 (2004), 245-253.

\bibitem{f} X. Fan, D. Zhao, {\it On the spaces $L^{p(x)}$and $W^{m,p(x)},$} J. Math. Anal. Appl. 263 (2011), 424-446.

\bibitem{g} A.M.Garsia, {\it Martingale inequalities: Seminar notes on recent progress,} Univ. of California, San
Diego, 1973.

\bibitem{h1} P. Harjulehto, P. Hasto, M. Pere,  {\it Variable
exponent Lebesgue spaces on metric spaces: The Hardy-Littlewood
maximal operator,} Real Anal. Exchange 30 (2004), 87-104.

\bibitem{h2} H. Hudzik, W. Kowalewski, {\it On some global and local geometry properties of
Musielak-Orlicz spaces,} Publ Math Debrecen 67 (2005), 41-64.

\bibitem{k} O. Kovacik, J. Rakosnik, {\it On spaces $L^{p(x)}$ and $W^{k,p(x)}$,} Czechoslovak Math. J. 41(1991), 592-618.

\bibitem{l1} E. Lenglart, D. Lepingle, M. Pratelli, {\it Presentation unifiee de certaines inegalites de la theorie des martigales,
,} Sem. Prob. 14, LNM 781, 1980, 26-48.

\bibitem{l2} A.K.Lerner, {\it Some remarks on the Hardy-Littlewood maximal
function on variable $L^p$ spaces, } Math. Z. 251 (2005), 509-521.

\bibitem{l3} R.L.Long, {\it Martingale spaces and inequalities,} Peking Univ Press,
Beijing. 1993.

\bibitem{mu} J.Musielak,  {\it Orlicz spaces and modular spaces,} LNM 1034, Springer-Verlag, Berlin, 1983.

\bibitem{n1} E.Nakai, G.Sadasue, {\it Maximal function on generalized martingale
Lebesgue spaces with variable exponent,}  Statis. Probab. Letters 83
(2013), 2167-2171.

\bibitem{n2} E.Nakai, G.Sadasue, {\it
Martingale Morrey-Campanato spaces and fractional integrals,} J.
Funct. Spaces Appl. 2012 (2012), Article ID 673929, 29 pages
http://dx.doi.org/10.1155/2012/673929

\bibitem{w} F.Weisz, {\it Martingale Hardy spaces and their applications in
Fourier analysis,}  LNM 1568, Springer-Verlag, Berlin, 1994.



\end{thebibliography}
\end{document}